\documentclass[a4paper]{article}			% Format de page
\usepackage[latin1]{inputenc}				% Pour les caract�res accentu�s
\usepackage[T1]{fontenc}					% Encodage de caract�res
\usepackage{lmodern}						% Police vectorielle Latin Modern
\usepackage[english]{babel}					% R�gles typographiques anglaises
\usepackage{amsmath,amsfonts}				% Symboles math�matiques
\usepackage{amssymb}						% Symboles math�matiques bis
\usepackage{stmaryrd}                       % Pour les intervalles d'entiers

\usepackage{bm}								% Pour les caracteres gras en mode mathematiques
\usepackage{amsthm}

\usepackage{a4wide} 						%Pour avoir des marges standard

\theoremstyle{plain}
\newtheorem{thm}{Theorem}[section]
\newtheorem{lem}[thm]{Lemma}
\newtheorem{prop}[thm]{Proposition}
\newtheorem{conj}[thm]{Conjecture}
\newtheorem{coro}[thm]{Corollary}

\theoremstyle{definition}
\newtheorem{defi}[thm]{Definition}
\newtheorem*{nota}{Notation}
\newtheorem{rem}[thm]{Remark}

\newcommand{\R}{\mathbb{R}}

\newcommand{\N}{\mathbb{N}}
\newcommand{\Z}{\mathbb{Z}}

\newcommand{\pr}{\mathbb{P}}

\title{Vertical shift and simultaneous Diophantine approximation on polynomial curves}
\author{Faustin ADICEAM\\
   Department of Mathematics, Logic House,\\ 
   National University of Ireland at Maynooth\\
   email~:\texttt{ fadiceam@gmail.com}}
\date{}

\makeatletter
\def\imod#1{\allowbreak\mkern10mu({\operator@font mod}\,\,#1)}
\makeatother

\newcommand\blfootnote[1]{					%Pour la classification MSC en bas de page
\begingroup
\renewcommand\thefootnote{}\footnote{#1}
\addtocounter{footnote}{-1}
\endgroup
}

\begin{document}
\maketitle
\begin{abstract}
The Hausdorff dimension of the set of simultaneously $\tau$--well approximable points lying on a curve defined by a polynomial $P(X) + \alpha$, where $P(X)\in\Z[X]$ and $\alpha \in \R$, is studied when $\tau$ is larger than the degree of $P(X)$. This provides the first results related to the computation of the Hausdorff dimension of the set of well approximable points lying on a curve which is not defined by a polynomial with integer coefficients. 

The proofs of the results also include the study of problems in Diophantine approximation in the case where the numerators and the denominators of the rational approximations are related by some congruential constraint.\blfootnote{MSC(2010)~: 11J83, 11K60}
\end{abstract}

\section{Introduction and statement of the results}

For any manifold $\mathcal{M} \subset \R^2$ and any real number $\tau >1$, denote by $W_{\tau}(\mathcal{M})$ the set of simultaneously $\tau$--well approximable points lying on $\mathcal{M}$, i.e. $$\widehat{W}_{\tau}(\mathcal{M}) = \left\{ (x,y)\in \mathcal{M} \; : \;  \left|x-\frac{p}{q}\right|<\frac{1}{q^{\tau}} \mbox{  and  } \left|y-\frac{r}{q}\right|<\frac{1}{q^{\tau}} \mbox{  i.o.} \right\}.$$ Here and in what follows, \textit{i.o.} stands for \textit{infinitely often}, that is, for infinitely many integers $p$, $r$ and $q$ with $q\ge 1$. 

Even in the simplest case where $\mathcal{M}$ is prescribed to be a planar curve defined by an equation with integer coefficients, the actual Hausdorff dimension $\dim \widehat{W}_{\tau}(\mathcal{M})$ of the set $\widehat{W}_{\tau}(\mathcal{M})$ may exhibit very different behaviours, although the starting point of the computation of the  dimension is generally the same~: it is shown that, if a pair of rationals $(p/q, r/q)$ realizes an approximation of $(x,y)\in\mathcal{M}$ at order $\tau$ as in the definition of the set $\widehat{W}_{\tau}(\mathcal{M})$, then for $\tau$ larger than some constant depending only on the curve, the point $(p/q, r/q)$  has to belong to $\mathcal{M}$ for $q$ large enough. Obviously, the assumption that $\mathcal{M}$ is a curve defined by some equation with \textit{integer} coefficients is then essential. The following two examples illustrate this fact.

\paragraph{}
Consider first, for any integer $l\ge 2$, the Fermat Curve $$\mathcal{F}_l := \left\{ (x,y) \in \R^2 \; : \;  x^l + y^l = 1 \right\}.$$ For $\tau >1$, let $(x,y)\in W_{\tau}(\mathcal{F}_l)$ and let $(p/q, r/q)$ be a pair of rational numbers such that $$x= \frac{p}{q} + \frac{\epsilon_x \theta_x}{q^{\tau}} \;\, \mbox{  and  } \;\, y= \frac{r}{q} + \frac{\epsilon_y \theta_y}{q^{\tau}}$$ with $\epsilon_x, \epsilon_y \in \{\pm 1 \}$ and $\theta_x , \theta_y \in (0,1)$. In particular, $p=O(q)$ and $r=O(q)$ as $q$ tends to infinity. On rearranging the equation $$q^l = \left(p + \frac{\epsilon_x \theta_x}{q^{\tau -1}} \right)^l + \left(r + \frac{\epsilon_y \theta_y}{q^{\tau -1}} \right)^l ,$$ it is readily seen that $$\left| q^{l} - p^l - r^l \right| \le \frac{C(l,x,y)}{q^{\tau-l}},$$ where $C(l,x,y,)$ is a strictly positive constant which depends on $x$, $y$ and $l$, but is independent of $q$. For $\tau > l$ and $q$ large enough, this implies that 
\begin{equation}\label{fermat}
q^l = p^l + r^l,
\end{equation} 
i.e. $(p/q, r/q) \in \mathcal{F}_l$. By Wiles' result on Fermat's Last Theorem \cite{wiles}, the latter equation is not solvable in positive integers as soon as $l \ge 3$. Therefore, if $(x,y)\in \widehat{W}_{\tau}(\mathcal{F}_l)$ ($l\ge 3$), then $(x,y) \in \left\{(1,0) ; (0,1) \right\}$ if $l$ is odd and $(x,y) \in \left\{(\pm 1,0) ; (0,\pm 1) \right\}$ if $l$ is even~: this means that $\widehat{W}_{\tau}(\mathcal{F}_l)$ contains at most four points if $\tau > l\ge 3$. 

In particular, this implies the following result~:

\begin{thm}\label{dimfermat}
For $l\ge 3$ and $\tau >l$, $$\dim \widehat{W}_{\tau}(\mathcal{F}_l) = 0 .$$
\end{thm}

\begin{rem}\label{hausdcircle}
If $l=2$, (\ref{fermat}) is soluble in infinitely many Pythagorean triples $(p,q,r)$ and the result of Theorem~\ref{dimfermat} is no longer true~: indeed, Dickinson and Dodson~\cite{ddods} have proved that $$\dim \widehat{W}_{\tau}(\mathcal{F}_2) = \frac{1}{\tau}$$ for $\tau >2$, which constituted the first reasonably complete non-trivial result for the Hausdorff dimension of the set $\widehat{W}_{\tau}(\mathcal{M})$ for a smooth manifold $\mathcal{M}$ in $\R^n$ when $\tau$ is larger than the extremal value of $1+1/n$. From their proof, it is also clear that the result holds true for any arc contained in $\mathbb{S}^1=\mathcal{F}_2$.
\end{rem}

\paragraph{}
Consider now the case where the manifold is an integer polynomial curve $$\Gamma = \lbrace (x, P(x))\in \R^2 \; : \; x\in \R\rbrace$$ in $\R^2$, where $P(X)\in\Z[X]$ is a polynomial of degree $d\ge 1$. Since Hausdorff dimension is unaffected under locally bi-Lipschitz transformations~\cite{falco}, it is not difficult to see that $\widehat{W}_{\tau}(\Gamma)$ ($\tau >0$) has the same Hausdorff dimension as the set $$W_{\tau}(P) := \left\{ x \in \R \; : \;  \left|x-\frac{p}{q}\right|<\frac{1}{q^{\tau}} \mbox{  and  } \left|P(x)-\frac{r}{q}\right|<\frac{1}{q^{\tau}} \mbox{  i.o.} \right\}.$$

Working with an appropriate Taylor expansion of $P(X)$, Budarina, Dickinson and Levesley~\cite{base} have proved that, for $\tau >d$, the only rational points which need to be taken into account for the computation of the Hausdorff dimension of the set $W_{\tau}(P)$ actually lie on the polynomial curve under consideration. Their result, which gave impetus to this paper, is the following (see~\cite{base} for the proof)~:

\begin{thm}[Budarina, Dickinson \& Levesley]\label{bbase}
For $\tau> \max\left(d, 2/d \right)$, the Hausdorff dimension of $W_{\tau}(P)$ is $$\dim W_{\tau}(P) = 		 	\frac{2}{d\tau}\cdotp$$
\end{thm}

In particular, for any $\tau>0$, the set $W_{\tau}(P)$ is always of positive Hausdorff dimension and therefore contains uncountably many points. 

\paragraph{}
The main result of this paper shows that this no longer holds true in the metric sense as soon as the curve $\Gamma$ is vertically translated by a real number. More precisely, given $\alpha \in \R$, let $W_{\tau}(P_{\alpha})$ denote the set of simultaneously $\tau$--approximable points lying on the polynomial curve $\Gamma_{\alpha} = \lbrace (x, P(x)+\alpha)\in \R^2 \; : \; x\in \R\rbrace$ in $\R^2$, that is, $$W_{\tau}(P_{\alpha}) = \left\{ x \in \R \; : \;  \left|x-\frac{p}{q}\right|<\frac{1}{q^{\tau}} \mbox{  and  } \left|P(x)+\alpha-\frac{r}{q}\right|<\frac{1}{q^{\tau}} \mbox{  i.o.} \right\}.$$ Then the main result proved in this paper reads as follows~:

\begin{thm}\label{princi}
Assume $d\ge 2$. If $\tau > d+1$, then $$W_{\tau}(P_{\alpha}) = \emptyset$$ for almost all $\alpha \in \R$. 
\end{thm}

Here as elsewhere, \textit{almost all} and \textit{almost everywhere} must be understood in the sense that the set of exceptions has Lebesgue measure zero.

Theorem~\ref{princi} improves a previous result due to Dickinson in~\cite{dettajapon} (Theorem~4), where the weaker bound $3d-1$ was proposed. The method developed in the proof of Theorem~\ref{princi} provides evidence that the bound $d+1$ in the above is in fact optimal. Indeed, it provides an upper bound for the Hausdorff dimension of $W_{\tau}(P_{\alpha})$ valid for almost all $\alpha\in\R$ and for $\tau \in(d, d+1]$ which vanishes when $\tau=d+1$~:

\begin{thm}\label{sec}
Assume $d\ge 2$. If $\tau \in (d, d+1]$, then 
\begin{align}\label{majorwt}
\dim W_{\tau}(P_{\alpha}) \le \frac{d+1-\tau}{\tau}
\end{align}
for almost all $\alpha \in \R$. 
\end{thm}

The relevance of the result of Theorem~\ref{sec} also appears clearly when compared with the following one, proved by Vaughan and Velani~\cite{planarbis}.
\begin{thm}[Vaughan \& Velani] \label{bdv}
Let $f$ be a three times continuously differentiable function defined on an interval $I$ of $\R$ and $\mathcal{C}_f := \left\{ (x, f(x))\in \R^2 \; : \; x\in I \right\}.$ Let $\tau \in [3/2 , 2)$ be given. Assume that $\dim \left\{ x\in I \; : \; f''(x)=0\right\}\le (3-\tau)/\tau .$ Denote by $W_{\tau}(f)$ the set of simultaneously $\tau$--well approximable points in $\R^2$ lying on the curve $\mathcal{C}_f$. Then $$\dim W_{\tau}(f) = \frac{3-\tau}{\tau} =: s.$$ Moreover, if $\tau \in (3/2 , 2)$, then the $s$--Hausdorff measure of the set $W_{\tau}(f)$ is infinite.
\end{thm}

Now if the degree of the polynomial $P(X)$ equals $d=2$, then the upper bound for the Hausdorff dimension of $W_{\tau}(P_{\alpha})$ given by~(\ref{majorwt}) for almost all $\alpha\in\R$ and for $\tau$ lying in the interval $(2, 3]$ has the same expression as the exact value of $\dim W_{\tau}(P_{\alpha})$ provided by Theorem~\ref{bdv}, which is valid for all $\alpha\in\R$ and for $\tau\in (3/2, 2)$.

\paragraph{}
Theorems~\ref{princi} and \ref{sec} seem to provide the first results related to the study of the Hausdorff dimension of the set of well approximable points lying on a curve which is not defined by a polynomial with integer coefficients. Besides this fact, the method involved in the proofs is also interesting in its own right since it includes the study of problems of Diophantine approximation by rationals whose numerators and denominators are related by some congruential constraint.  

It should also be emphasized that Theorems~\ref{princi} and \ref{sec} may easily be generalized to the case of a general decreasing approximating function $\Psi : \R_+ \rightarrow \R_+$ which tends to zero at infinity~: to this end, denote by $W_{\Psi}(P_{\alpha})$ the set of $\Psi$--well approximable points lying on the curve defined by the polynomial $P(X)+\alpha$ in such a way that $W_{\tau}(P_{\alpha})$ is the set $W_{\Psi}(P_{\alpha})$ with $\Psi(q)=q^{-\tau}$.  Let $\lambda_{\Psi}$ be the \emph{lower order} of $1/\Psi$, that is, $$\lambda_{\Psi}:= \liminf \frac{-\log \Psi (q)}{\log q} \textrm{ as } q \rightarrow \infty.$$ The lower order $\lambda_{\Psi}$ indicates the growth of the function $1/\Psi$ in a neighborhood of infinity. Note that this quantity is always positive since $\Psi$ tends to zero at infinity. With this notation at one's disposal, the generalization of 
Theorems~\ref{princi} and \ref{sec} may be stated as follows~:

\begin{coro}
Assume $d\ge 2$. If $\lambda_{\Psi} > d+1$, then $$W_{\Psi}(P_{\alpha}) = \emptyset$$ for almost all $\alpha \in \R$. Furthermore, if $\lambda_{\Psi}\in (d, d+1]$, then $$\dim W_{\Psi}(P_{\alpha}) \le \frac{d+1-\lambda_{\Psi}}{\lambda_{\Psi}}$$ for almost all $\alpha\in\R$.
\end{coro}

\begin{proof}
From the definition of the lower order $\lambda_{\Psi}$, it is readily verified that, for any $\epsilon >0$, $$\Psi(q)\le q^{-\lambda_{\Psi}+\epsilon} \textrm{ for all but finitely many } q\in\N^*.$$ Therefore, for any $\epsilon >0$, $$W_{\Psi}(P_{\alpha}) \subset W_{\lambda_{\Psi}-\epsilon}(P_{\alpha}).$$ The corollary then follows easily from Theorem~\ref{princi} and Theorem~\ref{sec}.
\end{proof}

\paragraph{}
The paper is organized as follows~: the problem of simultaneous Diophantine approximation under consideration is first reduced to a problem of Diophantine approximation concerning the quality of approximation of the real number $\alpha$ by rational numbers whose numerators and denominators are related by some congruential constraint (section~\ref{sec2}). The auxiliary lemmas collected in section~\ref{sec3} shall be needed in the course of the proofs of Theorem~\ref{princi} (section~\ref{sec4}) and Theorem~\ref{sec} (section~\ref{sec5}). Some remarks on the results and the method developed shall conclude the paper (section~\ref{sec6}).

For details about Hausdorff dimension and the proof of some of its basic properties which shall be used throughout, the reader is referred to~\cite{falco}.

\paragraph{}
Since the set $W_{\tau}\left( P_{\alpha}\right)$ is invariant when the real number $\alpha$ is translated by an integer, it shall be assumed throughout, without loss of generality, that $\alpha$ lies in the unit interval $[0,1]$. Once and for all, $P(X)\in\Z[X]$ is a fixed polynomial of degree $d\ge 2$ whose leading coefficient shall be denoted by $-a_d\in\Z^*$ for convenience.

\paragraph{\textbf{\large{Notation}}\\}
The following notation shall be used throughout~:

\begin{itemize}
\item $ \lfloor x \rfloor$ (resp. $\left\lceil{x}\right\rceil$), $x \in \R$~: the integer part of $x$ (resp. the smallest integer not less than $x$).
\item $(x)_{+} := \max\left\{0, x \right\}$ $(x\in\R)$.
\item $f \ll g$ (resp. $f \gg g$)~: notation equivalent to $f=O(g)$ (resp. $g=O(f)$).
\item $ \llbracket x , y \rrbracket$ ($x, y\in\R$, $x \le y$)~: interval of integers, i.e. $\llbracket x , y \rrbracket = \left\{ n\in\Z \; : \; x\le n \le y\right\}$.
\item $\lambda$~: the Lebesgue measure on the real line (or its restriction to the unit interval). 
\item $\textrm{Card}(X)$ or $|X|$~: the cardinality of a finite set $X$.
\item $A^{\times}$~: the set of invertible elements of a ring $A$.
\item $M^* = M\backslash \left\lbrace 0 \right\rbrace$ for any monoid $M$ with identity element $0$.
\item $\pr$ (resp. $\pi$, $\nu_{\pi}(q)$ for $q\in \N^*$)~: the set of prime numbers (resp. any prime number, the $\pi$--adic valuation of $q$).
\item $\varphi (n)$~: Euler's totient function.
\item $\tau (n)$~: the number of divisors of a positive integer $n$.
\item $\omega (n)$~: the number of distinct prime factors dividing an integer $n\ge 2$ ($\omega (1)=0$).
\item $\| f \|_{\infty}^{I}$~: the infinity norm of a continuous function $f$ over a bounded interval $I \subset \R$, i.e. $\| f \|_{\infty}^{I}= \underset{x\in I}{\sup}|f(x)|.$ 
\item $G_d\left(q\right)$ ($d,q \ge 1$ integers)~: the set of $d^{th}$ powers modulo $q$.
\item $aG_d\left(q\right) := \left\{ am \; : \; m\in G_d\left(q\right)\right\}$ ($d,q \ge 1$ integers, $a\in\Z/q\Z$).
\end{itemize}

\section{From the simultaneous case to Diophantine approximation under constraint} \label{sec2}

In this section, simultaneous approximation properties of a real number $x$ and of $P(x)+\alpha$ are linked to some properties of Diophantine approximation under a constraint of the real number $\alpha$, and conversely. The aforementioned constraint implies the resolution of a congruence equation involving the polynomial $P(X)$. This section is the key step to the proof of Theorems~\ref{princi} and \ref{sec}.

\subsection{Reduction of the problem}\label{reducprob}

Let $M$ be an integer and $W_{\tau}^{M}\left( P_{\alpha}\right) = W_{\tau}\left( P_{\alpha}\right) \cap \left[ M , M+1 \right]$, i.e. $$W_{\tau}^{M}\left( P_{\alpha}\right) = \left\{ x \in \left[ M , M+1 \right] \; : \;  \left|x-\frac{p}{q}\right|<\frac{1}{q^{\tau}} \mbox{  and  } \left|P(x)+\alpha-\frac{r}{q}\right|<\frac{1}{q^{\tau}} \mbox{  i.o.} \right\}.$$ It is clear that $$W_{\tau}\left( P_{\alpha}\right) = \bigcup_{M\in\Z} W_{\tau}^{M}\left( P_{\alpha}\right).$$ 

In order to compute the Hausdorff dimension of the set $W_{\tau}\left( P_{\alpha}\right)$, it is more convenient to first focus on the subsets $W_{\tau}^{M}\left( P_{\alpha}\right)$. To this end, the following two lemmas are needed. Recall that $d:=\deg P$.

\begin{lem}\label{sens1}
Let $\tau >0$ and $x \in \left[ M , M+1 \right]$ such that there exist rational numbers $p/q$ and $r/q$ satisfying $$
        x-\frac{p}{q}= \frac{\theta_x \epsilon_x}{q^{\tau}} \quad \mbox{ and } \quad
        P(x)+\alpha-\frac{r}{q} = \frac{\theta_y \epsilon_y}{q^{\tau}},
$$
with $\theta_x, \theta_y \in (0,1)$ and $\epsilon_x , \epsilon_y \in \left\{\pm 1 \right\}.$

Then $$\left|\alpha - \frac{rq^{d-1}-q^d P\left(\frac{p}{q} \right)}{q^d} \right| < \frac{K_M}{q^{\tau}},$$ where $K_M := 1+ \| P' \|_{\infty}^{\left[ M , M+1 \right]}.$
\end{lem}

\begin{proof}
The proof is straightforward~: first notice that $$\alpha+P(x)-P(x) - \frac{rq^{d-1}-q^d P\left(\frac{p}{q} \right)}{q^d} = \frac{\theta_y \epsilon_y}{q^{\tau}} - \left(P(x)-P\left(\frac{p}{q} \right) \right).$$ Now, by the Mean Value Theorem, there exists a point $c$ in $\left(x,\frac{p}{q} \right)$ such that $$P(x)-P\left(\frac{p}{q} \right)  = P'\left(c \right)\frac{\theta_x \epsilon_x}{q^{\tau}}\cdotp$$ Therefore, $$\alpha - \frac{rq^{d-1}-q^d P\left(\frac{p}{q} \right)}{q^d} = \frac{\theta_y \epsilon_y - P'\left(c \right)\theta_x \epsilon_x}{q^{\tau}},$$ which proves the lemma.
\end{proof}

The next result provides a partial converse to Lemma~\ref{sens1}. Here again, $K_M := 1+ \|P' \|_{\infty}^{\left[ M , M+1 \right]}$~:

\begin{lem}\label{sens2}
Let $b$ and $q\ge 1$ integers such that there exists an integer  $p\in \llbracket Mq , (M+1)q \rrbracket$ satisfying $$\frac{b}{q^d}+P\left(\frac{p}{q} \right) \in \frac{\Z}{q}:= \left\{\frac{a}{q}  \; : \; a\in \Z \right\}.$$ Assume furthermore that $$\alpha -\frac{b}{q^d} = \frac{\epsilon \theta K_M}{q^{\tau}},$$ where $\theta \in (0,1)$ and $\epsilon \in \left\{\pm 1 \right\}.$

Then there exists $r\in\Z$ such that for any $x \in \left(\frac{p}{q}-\frac{1}{q^{\tau}} ,\frac{p}{q}+\frac{1}{q^{\tau}}\right)$, $$\left|P(x)+\alpha - \frac{r}{q} \right| < \frac{2K_M}{q^{\tau}}\cdotp$$
\end{lem}

\begin{proof}
Let $r\in \Z$ be such that $$\frac{b}{q^d}+P\left(\frac{p}{q} \right) = \frac{r}{q}\cdotp$$
By the triangle inequality and the Mean Value Theorem, 
\begin{align*}
\left|P(x)+\alpha - \frac{r}{q} \right| & \le \left|P(x)-P\left(\frac{p}{q} \right)\right| + \left|\alpha - \frac{b}{q^d} \right|\\
&\le \frac{\| P' \|_{\infty}^{\left[ M , M+1 \right]}}{q^{\tau}} +\frac{\theta K_M}{{q^{\tau}} },
\end{align*}
hence the Lemma.
\end{proof}

\paragraph{}
For any integer $M$ and any real number $K>0$, let 
\begin{equation}\label{rtalphaM}
R_{\tau}^{M}(\alpha)\! \left[ K \right] := \left\{ x \in \left[ M , M+1 \right] \; : \;
\begin{split}
&  \left|x-\frac{p}{q}\right|<\frac{1}{q^{\tau}}\quad \mbox{  and  } \\
&  \left|\alpha-\frac{b}{q^d}\right|<\frac{K}{q^{\tau}}\quad \mbox{  with  } \quad\frac{b}{q^d}+P\left(\frac{p}{q} \right) \in \frac{\Z}{q}
\end{split}
\quad  \mbox{  i.o.} \right\}
\end{equation} 

and $$W_{\tau}^{M}\left( P_{\alpha}\right)\! \left[K \right] := \left\{ x \in \left[ M , M+1 \right] \; : \;  \left|x-\frac{p}{q}\right|<\frac{1}{q^{\tau}} \mbox{  and  } \left|P(x)+\alpha-\frac{r}{q}\right|<\frac{K}{q^{\tau}} \mbox{  i.o.} \right\}.$$ 
For simplicity, omit the square brackets in the above notation if $K=1$. 

With these definitions, Lemmas~\ref{sens1} and \ref{sens2} amount to claiming that, for any integer $M$, $$W_{\tau}^{M}\left( P_{\alpha}\right) \; \subset R_{\tau}^{M}(\alpha)\! \left[ K_M \right] \; \subset  W_{\tau}^{M}\left( P_{\alpha}\right)\! \left[2K_M \right].$$
Now it is readily seen that, for any $\epsilon >0$, the above inclusions imply that $$W_{\tau}^{M}\left( P_{\alpha}\right) \; \subset R_{\tau - \epsilon}^{M}(\alpha) \; \subset  W_{\tau - 2\epsilon}^{M}\left( P_{\alpha}\right).$$
Defining 
\begin{equation}\label{rtalpha}
R_{\tau}\left( \alpha \right) := \bigcup_{M\in\Z} R_{\tau}^{M}\left( \alpha \right),
\end{equation} 
it follows that, for any $\epsilon >0$, $$W_{\tau} \left( P_{\alpha}\right) \; \subset R_{\tau - \epsilon}(\alpha) \; \subset  W_{\tau - 2\epsilon}\left( P_{\alpha}\right).$$
Thus, the following proposition has been proved~:

\begin{prop}
For any $\tau >0$, $\dim W_{\tau} \left( P_{\alpha}\right) \le \underset{\epsilon \rightarrow 0^+}{\lim} \dim R_{\tau - \epsilon}\left( \alpha \right)$ . 

Furthermore, the equality $\dim W_{\tau} \left( P_{\alpha}\right) = \dim R_{\tau}\left( \alpha \right)$ holds at any point of continuity of the function $\tau \mapsto \dim W_{\tau}\left( P_{\alpha}\right).$
\end{prop}

Since the function $\tau \mapsto \dim W_{\tau}\left( P_{\alpha}\right)$ is obviously decreasing, it defines a regulated function (that is, it admits at every point both left and right limit). Now it is well--known that the set of discontinuities of a regulated function is at most countable, from which it follows that, for almost all $\tau >0$, $\dim W_{\tau}\left( P_{\alpha}\right) = \dim R_{\tau}\left( \alpha \right).$

In fact, much more may be expected. Defining the set $W_{\tau} \left( f \right) $ for any function $f$ in the same way as $W_{\tau} \left( P\right)$, one may indeed state this conjecture~:

\begin{conj}\label{conj}
For any smooth function $f$ defined over $\R$, the map $\tau \mapsto \dim W_{\tau} \left( f \right)$ is continuous.
\end{conj} 

Obviously, the statement may be extended both to higher dimensions and by weakening the assumption on the regularity of the function $f$. Note that in the case of simultaneous approximation of independent quantities, the dimension function is known to be continuous in any case (see~\cite{rynne} to specify this assertion). On the other hand, one cannot ask the function $\tau \mapsto \dim W_{\tau} \left( f \right)$ to be differentiable for any positive value of $\tau$ in the general case as shown by the example of the circle $\mathbb{S}^1$. Indeed, combining the multidimensional extension of Dirichlet's Theorem in Diophantine approximation, Remark~\ref{hausdcircle} and Theorem~\ref{hausdcircle}, it is possible to compute the value of $\dim W_{\tau}(\mathbb{S}^1)$ for any $\tau>0$~:

\begin{equation*}
\dim W_{\tau}(\mathbb{S}^1) = \left\{
    \begin{array}{ll}
        1 & \mbox{if } 0\le \tau \le 3/2 \\
        (3-\tau)/\tau & \mbox{if } 3/2 < \tau \le 2 \\
		1/\tau & \mbox{if } \tau > 2. 
    \end{array}
\right.
\end{equation*}
From Remark~\ref{hausdcircle}, this also holds true for any arc contained in $\mathbb{S}^1$. Thus, the function $\tau \mapsto \dim W_{\tau}(\mathbb{S}^1)$ is piecewise differentiable as a continuous piecewise rational function. It may be expected, as a generalization of Conjecture~\ref{conj}, that this behaviour holds true for any function $\tau \mapsto \dim W_{\tau}(f)$ provided that $f$ is \emph{regular enough}.

\paragraph{}
In what follows, Theorems~\ref{princi} and \ref{sec} shall be proven for the set $R_{\tau}\left( \alpha \right)$~: since the bounds provided by these theorems are continuous in $\tau$, it suffices to study $\dim R_{\tau}\left( \alpha \right)$ rather than dealing with $\dim R_{\tau - \epsilon}\left( \alpha \right)$ before letting $\epsilon$ tend to zero.

\subsection{The congruential constraint}\label{premier2}

The condition $$\frac{b}{q^d}+P\left(\frac{p}{q} \right) \in \frac{\Z}{q},$$ with $b$ and $p$ integers and $q$ a positive integer appears in the definition of the set $R_{\tau}\left( \alpha \right)$. Plainly, it amounts to the congruence equation
\begin{align}
b \equiv -q^d P\left(\frac{p}{q} \right)	\imod{q^{d-1}} \label{2}
\end{align}
having a solution. Since the reduction modulo $q$ of~(\ref{2}) is $$b\equiv a_dp^d \imod{q}$$ (recall that the leading coefficient of $P(X)$ is $-a_d$), it should be obvious that 
\begin{equation}\label{ensRdetravail}
\tilde{R}_{\tau}\left( \alpha \right) \subset R_{\tau}\left( \alpha \right) \subset R_{\tau}^*\left( \alpha \right)
\end{equation} 
for any $\tau>0$, where
\begin{equation}\label{rtalphaprime}
R_{\tau}^{*}(\alpha) := \left\{ x \in \R \; : \;
\begin{split}
&  \left|x-\frac{p}{q}\right|<\frac{1}{q^{\tau}}\quad \mbox{  and  } \\
&  \left|\alpha-\frac{b}{q^d}\right|<\frac{1}{q^{\tau}}\quad \mbox{  with  } \quad b\equiv a_dp^d \imod{q}
\end{split}
\quad  \mbox{  i.o.} \right\}
\end{equation}
and where the set $\tilde{R}_{\tau}\left( \alpha \right)$ is defined in the same way as $R_{\tau}\left( \alpha \right)$ in~(\ref{rtalphaM}) and (\ref{rtalpha}) with the additional constraint $\gcd(q, pda_d)=1$ on the denominators of the rational approximants.
 
In fact, the upper bound in Theorem~\ref{sec} shall be established in section~\ref{sec5} for the set $R_{\tau}^*\left( \alpha \right)$ whereas Theorem~\ref{princi} shall follow in an obvious way from the proof in subsection~\ref{subsec4.1} that the set 
\begin{equation}\label{Ietoiletau}
I^{*}_{\tau}(P):= \left\{\alpha \in (0,1) \; : \; 
  \left|\alpha-\frac{b}{q^d}\right|<\frac{1}{q^{\tau}}\quad \mbox{  with  } b\in a_d G_d(q) \mbox{  i.o.} \right\}
\end{equation} 
has zero Lebesgue measure when $\tau > d+1$ (recall that $G_d(q)$ denotes the set of $d^{th}$ powers modulo $q$). Furthermore, the bound $d+1$ given by Theorem~\ref{princi} cannot be trivially improved if it is shown that $I^{*}_{\tau}(P)$ contains a subset which is not of Lebesgue measure zero when $\tau \le d+1$.

To this end, it shall be proven in subsection~\ref{subsec4.2} that the subset $\tilde{I}_{\tau}(P) \subset I^{*}_{\tau}(P)$ has full measure whenever $\tau\le d+1$, where
\begin{equation}\label{Itildetau}
\tilde{I}_{\tau}(P):= \left\{\alpha \in (0,1) \; : \; 
  \left|\alpha-\frac{b}{q^d}\right|<\frac{1}{q^{\tau}}\quad \mbox{  with  } b\in a_d G_d^{\times}(q) \mbox{ and } \gcd(q,da_d)=1 \mbox{  i.o.} \right\},
\end{equation} 
and where $G_d^{\times}(q)$ denotes (with an abuse of notation) the set of \emph{primitive} $d^{th}$ powers modulo $q$.

It should be noted that $\tilde{I}_{\tau}(P)$ is to the set $\tilde{R}_{\tau}\left( \alpha \right)$ as $I^{*}_{\tau}(P)$ is to the set $R_{\tau}^{*}(\alpha)$ in the following sense~: assume that $b\equiv a_d\tilde{p}^d \imod{q}$ for some $\tilde{p}\in\Z$ satisfying $\gcd(q,\tilde{p}da_d)=1$ as in definition~(\ref{Itildetau}) of the set $\tilde{I}_{\tau}(P)$. From the Chinese Remainder Theorem, solving this congruence equation modulo $q$ amounts to solving the same equation modulo $\pi^{\nu_{\pi}(q)}$ for any prime divisor $\pi$ of $q$. Now, under the assumption that $\gcd(q,\tilde{p}da_d)=1$, any solution $\tilde{p}$ of $b\equiv a_d\tilde{p}^d \imod{\pi^{\nu_{\pi}(q)}}$ may be lifted, thanks to Hensel's lemma, to a unique solution $p$ of the congruence equation~(\ref{2}) taken modulo $\pi^{\nu_{\pi}(q)(d-1)}$ ($d\ge 2$) such that $\pi$ does not divide the product $pda_d$. Therefore, using once again the Chinese Remainder Theorem, a solution in $\tilde{p}$ to $b\equiv a_d\tilde{p}^d \imod{q}$ satisfying $\gcd(q,\tilde{p}da_d)=1$ may be lifted in a unique way to a solution $p$ of Equation~(\ref{2}) such that $\gcd(q,pda_d)=1$ as in the definition of the set $\tilde{R}_{\tau}\left( \alpha \right)$.

\section{Some auxiliary lemmas} \label{sec3}

In this section are collected various results which shall be needed later.

\subsection{Comparative growths of some arithmetical functions}

For $n\ge 2$ an integer, let $\tau (n)$ be the number of divisors of $n$ and $\omega (n)$ the number of \textit{distinct} prime factors dividing $n$~: if $n= \prod_{i=1}^r \pi_{i}^{\nu_i}$ is the prime factor decomposition of this integer, recall that
\begin{align*}
\omega(n) = r \quad \mbox{ and } \quad \tau(n) = \prod_{i=1}^r \left( \nu_i +1 \right).
\end{align*}

Some results about comparative growth properties of these two arithmetical function are now recalled.

\begin{lem}\label{tau}
For any $\epsilon >0$, $\tau(n)=o\left(n^ \epsilon \right)$ and the average value of $\tau (n)$ is $\log n$, i.e. $$\frac{1}{n}\sum_{k=1}^{n} \tau(k) \underset{n\rightarrow +\infty}{\sim} \log n.$$
\end{lem}

\begin{proof}
See~\cite{hw}, Theorems~315 and 320.
\end{proof}

As is well--known, the average value of $\omega(n)$ is asymptotic to $\log \log n$ when $n$ tends to infinity (\cite{hw}, \S 22.11). However, a stronger statement similar to Lemma~\ref{tau} shall be needed in the proofs to come. To this end, the definition of the maximal order of an arithmetical function is introduced~:

\begin{defi}
An arithmetical function $f$ has \textit{maximal} (resp. \textit{minimal}) \textit{order $g$} if $g$ is a positive nondecreasing arithmetical function such that $$\underset{n\rightarrow +\infty}{\limsup}\, \frac{f(n)}{g(n)} = 1 \quad \left( \textrm{resp.} \; \underset{n\rightarrow +\infty}{\liminf} \,\frac{f(n)}{g(n)} = 0 \right).$$
\end{defi}

For instance, it is not difficult to see that the identity function is both a minimal and a maximal order for Euler's totient function.

\begin{lem}\label{omega}
A maximal order for $\omega (n)$ is $\log n / \log \log n$.

In particular, for any $\epsilon >0$ and any positive integer $m$, $$\omega(n)=o\left( \log n \right)\quad \textrm{ and } \quad m^{\omega(n)}=o\left(n^\epsilon \right).$$
\end{lem}

\begin{proof}
The first result is implicit in~\cite{hw}, p.355. The others follow easily from this one.
\end{proof}

\subsection{Counting the number of power residues in a reduced system of residues}\label{countt}

The congruence equations appearing in subsection~\ref{premier2} in the definition of the sets $\tilde{I}_{\tau}(P)$ and $I^{*}_{\tau}(P)$ on the one hand and $\tilde{R}_{\tau}(\alpha)$ and $R_{\tau}^{*}(\alpha)$ on the other involve power residues modulo an integer $q\ge 1$. The cardinality of such a set is now computed.

\paragraph{}
Let $n\ge 2$ and $d\ge 2$ be integers. Denote by $r_d(n)$ (resp. by $e_d(n)$) the number of distinct $d^{th}$ powers in the system of residues modulo $n$ (resp. in the \emph{reduced} system of residues modulo $n$) and by $u_d(n)$ the number of $d^{th}$ roots of unity modulo $n$, that is,
\begin{align*}
r_d(n) &= \textrm{Card} \left\{ m^d \imod{n} \; : \; m \in \Z/n\Z \right\},  \\
e_d(n) &= \textrm{Card} \left\{ m^d \imod{n} \; : \; m \in \left(\Z/n\Z 
\right)^{\times} \right\},  \\
u_d(n)& = \textrm{Card} \left\{ m\in \Z/n\Z\; : \; m^d \equiv 1 \imod{n} \right\} .
\end{align*}
Set furthermore $r_d(1)=e_d(1)=u_d(1)=1$.

\begin{rem}\label{multirdq}
If $u(f,n)$ denote the number of solutions in $x$ of the congruence $$f(x):=\sum_{k=0}^{d} a_k x^k \equiv 0 \imod{n}$$ for a given polynomial $f\in \Z [X]$ of degree $d\ge 1$, it is well--known that, as a consequence of the Chinese Remainder Theorem, $u(f,n)$ is a multiplicative function of $n$. It follows that $u_d(n)$ is multiplicative with respect to $n$ for any fixed $d$.

In fact, the same holds true for $r_d(n)$ and $e_d(n)$~:
\end{rem}

\begin{lem}\label{multiruedq}
For any fixed $d$, the functions $r_d(n)$, $e_d(n)$ and $u_d(n)$ are multiplicative with respect to $n$.
\end{lem} 

\begin{proof}
See~\cite{koro}, Lemma~1 for the case of the functions $r_d(n)$ and $e_d(n)$.
\end{proof}

Explicit formulae may be given for $r_d(n)$, $e_d(n)$ and $u_d(n)$. Since these arithmetical functions are multiplicative when $d$ is fixed, it suffices to give such formulae in the case where $n$ is a power of a prime.

\begin{prop}\label{decompte}
Let $n=\pi^k$ be a power of a prime number ($\pi \in \pr$, $k\ge 1$ integer). Then the following equations hold~: $$e_d(n)=\frac{\varphi(\pi^k)}{u_d(\pi^k)} \quad and \quad r_d(n) = \frac{\varphi(\pi^k)}{u_d(\pi^k)} + \frac{\varphi(\pi^{k-d})}{u_d(\pi^{k-d})}+ \dots + \frac{\varphi(\pi^{k-md})}{u_d(\pi^{k-md})}+1,$$ where $m$ stands for the largest integer such that $k-md\ge 1$.

Furthermore,
$$
u_d(n)=\left\{
    \begin{array}{ll}
        \gcd(2d, \varphi(n)) & \mbox{if } 2|d,\; \pi=2 \mbox{ and } k\ge 3 ,\\
        \gcd(d, \varphi(n)) & \mbox{otherwise.}
    \end{array}
\right.$$
\end{prop}

\begin{proof}
See~\cite{koro}, Lemmas~2 and 3.
\end{proof}

\begin{rem}\label{remsolutionclasse}
Consider a partition of all numbers in the complete system of residues modulo $\pi^k$ ($\pi \in \pr$, $k\ge 1$ integer) into classes with regard to their divisibility by $\pi^s$ and not $\pi^{s+1}$, that is, the numbers of the form $x\pi^s$ with $\gcd(x,\pi)=1$ belong to the class numbered $s$, $0\le s \le k$. As is made clear from the proof of Proposition~\ref{decompte} in~\cite{koro}, the quantity $\varphi(\pi^{k-sd})/u_d(\pi^{k-sd})$ with $k-sd\ge 1$ counts the number of distinct elements modulo $\pi^k$ obtained when taking the $d^{th}$ power of the numbers in the $s^{th}$ class. If $sd\ge k$, then the $d^{th}$ power of any element in the $s^{th}$ class is equal to zero modulo $\pi^k$.

Furthermore, the proof of Proposition~\ref{decompte} also implies that, if $k-sd\ge 1$ and if $b \imod{\pi^k}$ is the $d^{th}$ power of an element in the $s^{th}$ class, then the number of solutions in $x$ to the congruence equation $b \equiv x^d \imod{\pi^k}$ is precisely $u_d(\pi^{k-sd})$.
\end{rem}

\section{The set $W_{\tau}(P_{\alpha})$ when $\tau > d+1$} \label{sec4}

Theorem~\ref{princi} is now proved and the optimality of the lower bound $d+1$ appearing in this theorem is also studied.

\subsection{Emptiness of the set for almost all $\alpha\in\R$}\label{subsec4.1}

In order to establish the result of Theorem~\ref{princi}, recall that from the discussion held in subsection~\ref{reducprob} and from the inclusions~(\ref{ensRdetravail}), it suffices to prove that the set $R_{\tau}^{*}(\alpha)$ as defined in~(\ref{rtalphaprime}) is empty in the metric sense when $\tau > d+1$. This in turn follows from the fact that, as a consequence of the convergent part of the Borel--Cantelli lemma, the set $I^{*}_{\tau}(P)$ as defined in~(\ref{Ietoiletau}) satisfies the same property.

To see this, first notice that, for any $N\ge 1$, a cover of $I^{*}_{\tau}(P)$ is given by $\cup_{q\ge N} J^*_{\tau}(q),$ where
\begin{align}\label{jqetoiletau}
J^*_{\tau}(q) := \bigcup_{\underset{b\in a_d G_d(q)}{0\le b \le q^d-1}}\left(\frac{b}{q^d}-\frac{1}{q^{\tau}} , \frac{b}{q^d}+\frac{1}{q^{\tau}}  \right).
\end{align}  
If $\tau>d$ and $q\ge 1$ is large enough, $J^*_{\tau}(q)$ is a union of $\left|a_dG_d(q) \right|q^{d-1}$ non--overlapping intervals, each of length $2/q^{\tau}$, that is, 
\begin{equation}\label{mesurejtquetoile}
\lambda\left(J^*_{\tau}(q) \right) = \frac{2\left|a_dG_d(q) \right|q^{d-1}}{q^{\tau}},
\end{equation} 
where $\lambda$ denotes the Lebesgue measure on the real line. On the other hand, since the ring $a_d\Z/q\Z$ is isomorphic to $\Z/\tilde{q}\Z$, where $\tilde{q} = q/\gcd(q,a_d)$, the following relationships hold true~: 
\begin{equation}\label{encadrementlambdarq}
r_d\left(\tilde{q}\right) := \left| G_d\left(\tilde{q} \right) \right| = \left| a_d G_d(q)\right| \le \left| G_d(q) \right|=: r_d(q).
\end{equation}
In order to study the convergence of the series $\sum_{q\ge 1} \lambda\left(J^*_{\tau}(q) \right)$, an upper bound (resp. a lower bound) for $r_d(q)$ (resp. for $r_d(\tilde{q})$) shall be established. Regarding the upper bound for $r_d(q)$, Lemma~\ref{multiruedq} and Proposition~\ref{decompte} imply that $$r_d(q) = \prod_{\stackrel{\pi |q}{\pi\in\pr}} \left(1+\sum_{s=0}^{m_q(\pi,d) } \frac{\varphi\left(\pi^{\nu_{\pi}(q)-sd} \right)}{u_d\left(\pi^{\nu_{\pi}(q)-sd} \right)} \right),$$ where $m_q(\pi,d) := \left\lfloor \frac{\nu_{\pi}(q)-1}{d}\right\rfloor$. Now, it is easily checked that, for all $s \in \llbracket 0, m_q(\pi,d) \rrbracket$, $$\frac{\varphi\left(\pi^{\nu_{\pi}(q)-sd} \right)}{u_d\left(\pi^{\nu_{\pi}(q)-sd} \right)}\le \frac{\varphi\left(\pi^{\nu_{\pi}(q)} \right)}{u_d\left(\pi^{\nu_{\pi}(q)} \right)},$$ hence 
\begin{align}\label{majorationrdq}
r_d(q) &\le \prod_{\stackrel{\pi |q}{\pi\in\pr}} \left[1+ \left(1+ \frac{\nu_{\pi}(q)-1}{d} \right) \frac{\varphi\left(\pi^{\nu_{\pi}(q)} \right)}{u_d\left(\pi^{\nu_{\pi}(q)} \right)} \right] \nonumber \\
&\le 2^{\omega(q)} \frac{\varphi(q)}{u_d(q)} \prod_{\stackrel{\pi |q}{\pi\in\pr}} \left(1+\nu_{\pi}(q) \right) = 2^{\omega(q)} \frac{\varphi(q)}{u_d(q)} \tau(q) \nonumber \\
&\le 2^{\omega(q)}\tau(q) q.
\end{align} 

As for the lower bound for $r_d(\tilde{q})$, first notice that Lemma~\ref{multiruedq} and Proposition~\ref{decompte} lead to the estimate 
\begin{equation}\label{majudq}
1\le u_d(q) \le (2d)^{\omega(q)}
\end{equation} 
valid for all $q\ge 1$. One may then deduce from them that 
\begin{align}\label{rqtilde}
r_d\left(\tilde{q} \right) &\ge e_d\left(\tilde{q} \right) = \frac{\varphi\left(\tilde{q}\right)}{u_d\left(\tilde{q}\right)} =\frac{\tilde{q}}{u_d\left(\tilde{q}\right)} \prod_{\stackrel{\pi |\tilde{q}}{\pi\in\pr}} \left(1-\frac{1}{\pi} \right) \nonumber \\
&\ge \frac{\tilde{q}}{\left(4d\right)^{\omega\left(\tilde{q}\right)}} \\
&\ge \frac{q}{\left|a_d\right| \left(4d\right)^{\omega\left(q\right)}} \nonumber,
\end{align}
the last inequality following from the definition of $\tilde{q}$.

Finally, the combination of the relationships~(\ref{mesurejtquetoile}), (\ref{encadrementlambdarq}), (\ref{majorationrdq}) and (\ref{rqtilde}) leads to the inequalities
\begin{equation}\label{encadrementserielambda}
\sum_{q\ge 1} \frac{1}{(4d)^{\omega(q)}q^{\tau-d}} \; \ll \; \sum_{q\ge 1} \lambda\left(J^*_{\tau}(q) \right) \; \ll \; \sum_{q\ge 1} \frac{2^{\omega(q)}\tau(q)}{q^{\tau-d}}\cdotp
\end{equation} 
From Lemmas~\ref{tau} and \ref{omega}, the right-hand side converges for any $\tau>d+1$, hence $\lambda(I^{*}_{\tau}(P))=0$ for $\tau > d+1$. This bound is best possible according to the convergent part of the Borel--Cantelli lemma since the series $\sum_{q\ge 1} \lambda\left(J^*_{\tau}(q) \right)$ diverges for $\tau \le d+1$. This is indeed implied by~(\ref{encadrementserielambda}) and the following general lemma.

\begin{lem}\label{serieL}
Let $n$ be a positive integer and $z$ be a positive real number. Define for any positive real number $s$ the series $$L_z (s):= \sum_{\underset{\gcd(q,n)=1}{q\ge 1}}\frac{z^{\omega(q)}}{q^s}\cdotp$$

Then the series $L_z (s)$ converges if, and only if, $s>1$.
\end{lem}

\begin{proof}
Let $\chi_{n}$ be the Dirichlet principal character modulo $n$, i.e. for an integer $q\ge 1$, 
$$
\chi_n(q) = \left\{
    \begin{array}{ll}
        1 & \mbox{if } \gcd(n,q)=1\\
        0 & \mbox{otherwise.}
    \end{array}
\right.
$$

Then $$L_z (s)= \sum_{q\ge 1}\frac{\chi_n(q) z^{\omega(q)}}{q^s}\cdotp$$ Since $\chi_n(q)z^{\omega(q)}$ is a multiplicative arithmetical function, $L_z (s)$ admits an Euler product expansion given by 
\begin{align}\label{fctl}
L_z (s) = \prod_{\pi\in\pr}\left( 1+\sum_{k=1}^{+\infty}\frac{\chi_{n}(\pi)z^{\omega(\pi)}}{\pi^{ks}}\right) = \prod_{\underset{\gcd(\pi,n)=1}{\pi\in\pr}}\left(1+\frac{z}{\pi^s-1} \right).
\end{align} 
Since only positive  quantities are considered, $L_z(s)$ converges if, and only if, the right--hand side of~(\ref{fctl}) converges. Taking the logarithm of these equations, $L_z(s)$ is seen to converge if, and only if, $$\sum_{l \in \left(\Z/n\Z\right)^{\times}} \sum_{\underset{\pi\equiv  l\! \imod{n}}{\pi \in \pr}}\frac{1}{\pi^s}$$ converges, which is the case if, and only if, for all $l \in \left(\Z/n\Z\right)^{\times}$, $$\sum_{\underset{\pi\equiv l \! \imod{n}}{\pi \in \pr}}\frac{1}{\pi^s}$$ converges. By Dirichlet's theorem on arithmetic progressions, for all $l \in \left(\Z/n\Z\right)^{\times}$, $$\sum_{\underset{\pi\equiv l \!\imod{n}}{\pi \in \pr}}\frac{1}{\pi^s} \; \underset{s \rightarrow 1^+}{\sim}\;\frac{1}{\varphi(n)}\log\left(\frac{1}{s-1}\right).$$ This completes the proof.
\end{proof}

\subsection{Optimality of the lower bound $d+1$}\label{subsec4.2}

The divergence of the series $\sum_{q\ge 1} \lambda\left(J^*_{\tau}(q) \right)$ for $\tau \le d+1$ does not guarantee that the set $I^{*}_{\tau}(P)$ is not of Lebesgue measure zero, in which case the bound $d+1$ appearing in the statement of Theorem~\ref{princi} could be trivially improved. This problem is now tackled by showing, as mentioned in the discussion held in subsection~\ref{premier2}, that the subset $\tilde{I}_{\tau}(P)$ of $I^{*}_{\tau}(P)$ as defined by~(\ref{Itildetau}) has full measure whenever $\tau\le d+1$. 

To this end, the author considered in~\cite{extensionduffschaeff} the Theorem of Duffin and Schaeffer in Diophantine approximation~\cite{dskp}, which generalizes the classical theorem of Khintchine to the case of any error function under the assumption that all the rational approximants are irreducible. He extended it to the case where the numerators and the denominators of the rational approximants were related by a congruential constraint stronger than coprimality (\cite{extensionduffschaeff}, Theorem~1.2). As a corollary of this extension, setting for all integers $q\ge 1$ 
\begin{equation}\label{sdq}
s_d(q):= \frac{e_d(q)}{q}
\end{equation} 
(see Lemma~\ref{multiruedq} and Proposition~\ref{decompte} for an expression of $e_d(q)$), the following result was also obtained~:

\begin{thm}(\cite{extensionduffschaeff}, Corollary~1.4)\label{varduffinschaeffer}
Let $(q_k)_{k\ge 1}$ be a strictly increasing sequence of positive integers and let $(\alpha_k)_{k\ge 1}$ be a sequence of positive real numbers. Assume that~:
\begin{align*}
& \textbf{(a)} \quad \sum_{k=1}^{+\infty} \alpha_{k} = +\infty , \\
& \textbf{(b)} \quad \sum_{k=1}^{n} \alpha_k s_d\left( q_k^d \right) > c\sum_{k=1}^{n} \alpha_k \;\; \textrm{ for infinitely many positive integers } n \textrm{ and a real number } c>0, \\
& \textbf{(c)} \quad \gcd(q_k , a_d)=1 \;\; \textrm{ for all } k\ge 1. 
\end{align*}
Then for almost all $\alpha \in\R$, there exist arbitrarily many relatively prime integers $b_k$ and $q_k$ such that 
\begin{align*}
\left|\alpha - \frac{b_k}{q_k^d} \right| < \frac{\alpha_k}{q_k^d}\quad \textrm{and} \quad b_k\in a_d G_d^{\times}(q_k),
\end{align*}
where $G_d^{\times}(q_k)$ was defined at the same time as the set $\tilde{I}_{\tau}(P)$ by~(\ref{Itildetau}).
\end{thm}

One can deduce from Theorem~\ref{varduffinschaeffer} a stronger result than the one required to prove that the set $\tilde{I}_{\tau}(P)$ has full Lebesgue measure when $\tau \le d+1$~:

\begin{coro}\label{infinit}
Let $s\in (0,1]$ and let $m$ be a positive integer.

Then for almost all $\alpha\in\R$, there exist infinitely many integers $q$ and $b$, $q\ge 1$, satisfying 
\begin{center}
\begin{tabular}{c l}
\textbf{(i)} &\quad $\left|\alpha - \frac{b}{q^d} \right| < \frac{1}{q^{d+s}}$ , \rule[-7pt]{0pt}{20pt}\\
%&\\
\textbf{(ii)} &\quad $b\in a_d G_d^{\times}(q)$ , \rule[-7pt]{0pt}{20pt}\\
%&\\
\textbf{(iii)} &\quad $\gcd(q, da_d) =1$ , \rule[-7pt]{0pt}{20pt}\\
%&\\
\textbf{(iv)}  &\quad $\omega(q) \le m$. \rule[-7pt]{0pt}{20pt}\\
\end{tabular}
\end{center}

In particular, $\lambda(\tilde{I}_{\tau}(P)) = \lambda(I^{*}_{\tau}(P))=1$ when $\tau \le d+1$.
\end{coro}

\begin{proof}
Maintaining the notation of Theorem~\ref{varduffinschaeffer}, choose for the sequence $\left(q_k \right)_{k\ge 1}$ the successive elements of the set $\left\{n \in \N^* \; : \; \gcd(n, da_d)=1 \;\mbox{ and }\; \omega(n)\le m \right\}$ ordered increasingly and for $\left(\alpha_k\right)_{k\ge 1}$ the sequence $\left(1/q_k^s\right)_{k\ge 1}$.

Then $$\sum_{k\ge 1} \alpha_{k} \ge \sum_{\underset{\pi \nmid da_d}{\pi\in\pr}}\frac{1}{\pi^s}$$ and the right--hand side is a divergent series for $s\in (0,1]$. Furthermore, from~(\ref{majorationrdq}) and~(\ref{sdq}) on the one hand and from the choice of the sequence $\left(q_k \right)_{k\ge 1}$ on the other, for any positive integer $k$, $$s_{d}(q_k^d) = \frac{\varphi\left(q_k^d\right)}{q_k^d \, u_d\left(q_k^d\right)} \ge \frac{1}{(4d)^{m}}>0.$$

Theorem~\ref{varduffinschaeffer} completes the proof.
\end{proof}

\begin{rem}
It is not difficult to see that, for almost all $\alpha \in \R$, the sequence of denominators $(q_k)_{k\ge 1}$ in Corollary~\ref{infinit} may be chosen in such a way that (i), (ii) and (iii) hold and such that the sequence $\left(\omega\left(q_k \right)\right)_{k\ge 1}$ is unbounded. Indeed, define first for any positive integer $m$ the sequence $(n_{m,k})_{k\ge 1}$ as being the sequence of the successive elements of the set $$\left\{n \in \N^* \; : \; \gcd(n, da_d)=1 \;\mbox{ and }\; \omega(n)= m \right\}$$ ordered increasingly. Let $(\alpha_{m,k})_{k\ge 1}$ be the sequence $(1/n_{m,k}^s)_{k\ge 1}$, where $s\in (0,1]$, and let $$D_m := \left\{\alpha \in (0,1] \; : \; (i), (ii) \textrm{ and } (iii) \textrm{ hold true with } \omega(q) = m,\, i.o\right\}.$$ 

Denote by $(\pi_i)_{i\ge 1}$ the increasing sequence of primes. Since for all $k\ge 1$, $$s_d(n_{m,k}^d)  \ge \frac{1}{(4d)^{m}}>0 \;\textrm{ and }\; \sum_{k\ge 1} \alpha_{m,k} \, \ge \sum_{\underset{\pi_{i_j} \nmid 2da_d}{1\le i_1 < \dots < i_m }}\frac{1}{(\pi_{i_1}\dots \pi_{i_m})^s},$$
which is a divergent series, a similar reasoning to that of the proof of Corollary~\ref{infinit} shows that $\lambda(D_m)=1$ for any $m\in\N^*$. Then $\lambda\left(\cap_{m\ge 1} D_m \right) = 1$, hence in particular the result.
\end{rem}

\section{Upper bound for the Hausdorff dimension of $W_{\tau}(P_{\alpha})$ when $\tau$ lies in the interval $(d , d+1]$}\label{sec5}

Theorem~\ref{sec} shall be proven in this section after the study of the asymptotic behavior of the number of solutions of Diophantine inequalities.

\subsection{Asymptotic behavior of the number of solutions of Diophantine inequalities}\label{111}

Given a sequence of intervals $\left(I_q\right)_{q\ge 1}$ inside the unit interval and a real number $\alpha$, let $\mathcal{N}\left(Q, \alpha \right)$ denote the number of integers $q\le Q$ such that $q\alpha \in I_q \imod{1}$, that is,
\begin{align}\label{nqalpha}
\mathcal{N}\left(Q, \alpha \right) := \textrm{Card} \left\{q \in \llbracket 1 , Q \rrbracket\; : \; q\alpha \in I_q \imod{1} \right\}.
\end{align}
The asymptotic behavior of $\mathcal{N}\left(Q, \alpha \right)$ as $Q$ tends to infinity has been studied by Sprind{\v{z}}uk who exploited ideas from the works of W. Schmidt and H. Rademacher in the theory of orthogonal series (see~\cite{sprin} for further details).

\begin{thm}(\cite{sprin}, Theorem~18)\label{estimnaalphaerreur}
Let $\left(I_q\right)_{q\ge 1}$ be a sequence of intervals inside the unit interval $[0, 1]$ such that $$\sum_{q=1}^{+\infty} \lambda\left(I_q \right) = +\infty.$$
For any real number $\alpha$, define $\mathcal{N}\left(Q, \alpha \right)$ as in~(\ref{nqalpha}).

Then, for almost all $\alpha \in \R$, $$\mathcal{N}\left(Q, \alpha \right) = \Phi (Q) + O\left(\sqrt{\Psi(Q)} \left(\log \Psi (Q) \right)^{3/2+\kappa} \right),$$ where
\begin{align*}
\Phi (Q) := \sum_{q=1}^{Q} \lambda\left( I_q\right), \quad \Psi (Q):= \sum_{q=1}^{Q} \lambda\left( I_q\right)\tau(q)
\end{align*}
and $\kappa >0$ is arbitrary.
\end{thm}

The notation of Theorem~\ref{estimnaalphaerreur} is maintained in the next corollary.

\begin{coro}\label{equivnqalpha}
Under the assumptions of Theorem~\ref{estimnaalphaerreur}, suppose that one of the following conditions holds~:
\begin{align*}
& \textbf{(i)} \quad \Phi (Q) \gg Q^{\delta} \textrm{ for all } Q>0,\textrm{ for some } \delta >0 .\\
& \textbf{(ii)} \quad \lambda\left(I_q \right) \textrm{ decreases monotonically and } \Phi(Q) \gg \left( \log Q \right)^{1+\delta} \textrm{ for all } Q>0,\textrm{ for some } \delta >0.
\end{align*}
Then $$\mathcal{N}\left(Q, \alpha \right) \sim \sum_{q=1}^{Q} \lambda \left(I_q \right) \quad \mbox{ as } Q \rightarrow +\infty.$$
\end{coro}

\begin{proof}
If condition~(i) holds, then the result is a simple consequence of Theorem~\ref{estimnaalphaerreur} and the fact that $\tau(q) \ll q^{\epsilon}$ for any $\epsilon >0$ (Lemma~\ref{tau}).

If condition~(ii) holds, since $\sum_{1\le k\le q} \tau(k) \ll q\log q$ by Lemma~\ref{tau}, making an Abel transformation in the expression for $\Psi(Q)$ shows that $\Psi(Q) \ll \Phi(Q) \log Q$. The conclusion follows in this case also.
\end{proof}

\begin{rem}\label{sssuite}
In the statement of Theorem~\ref{estimnaalphaerreur}, no restrictions whatsoever are imposed on the way the intervals $I_q$ vary with $q$. Therefore the condition $q\alpha \in I_q \imod{1}$ appearing in the definition~(\ref{nqalpha}) of $\mathcal{N}\left(Q, \alpha \right)$ may be regarded as holding for the numbers $q_k$ of an arbitrarily increasing sequence. Then Corollary~\ref{equivnqalpha} is still valid for such a sequence $\left(q_k\right)_{k\ge 1}$.
\end{rem}

\subsection{The proof of Theorem~\ref{sec}}

In order to prove Theorem~\ref{sec}, recall that it suffices to establish the upper bound for the Hausdorff dimension of $W_{\tau}\left(P_{\alpha} \right)$ in the case of the set $R^{*}_{\tau}(\alpha)$ as defined in~(\ref{rtalphaprime}). Without loss of generality, it may be assumed that $\tau\in (d, d+1)$, the result in the case $\tau=d+1$ following from an obvious passage to the limit. Furthermore, since the set $R^{*}_{\tau}(\alpha)$  is invariant when translated by an integer, it suffices to prove Theorem~\ref{sec} for the subset $R^{*}_{\tau}(\alpha) \cap [0, 1]$ which, for the sake of simplicity, shall still be denoted by $R^{*}_{\tau}(\alpha)$ in what follows.

The fact that the fractions $p/q$ are not necessarily irreducible in the definition of the set $R^{*}_{\tau}(\alpha)$ induces considerable difficulties as one needs to take into account the order of magnitude of the highest common factor between $p$ and $q$ to compute $\dim R^{*}_{\tau}(\alpha)$. In fact, it is more convenient to work with $\gcd(b,q)$. To this end, define for $\epsilon\in [0, 1]$ and $\delta >0$ the set $R_{\tau}^{*}(\alpha, \epsilon, \delta)$ as
\begin{equation}\label{rtaualhpaetoileepsilondelta}
\left\{ x \in [0,1] \; : \;
\begin{split}
&  \left|x-\frac{p}{q}\right|<\frac{1}{q^{\tau}}\quad \mbox{  and  } \\
&  \left|\alpha-\frac{b}{q^d}\right|<\frac{1}{q^{\tau}}\quad \mbox{  with  } \quad b\equiv a_dp^d \imod{q}
\end{split}
\quad \mbox{ i.o. with } q^{\epsilon}\le \gcd(b,q) < q^{\epsilon + \delta} \right\}.
\end{equation} 
It should be obvious that 
\begin{equation*}
R^{*}_{\tau}(\alpha) = \bigcup_{0\le \epsilon < \epsilon + \delta \le 1} R_{\tau}^{*}(\alpha, \epsilon, \delta).
\end{equation*}

Let furthermore $I_{\tau}^{*}(\alpha, \epsilon, \delta)$ be the set 
\begin{equation}\label{Ietoiletauepsilondelta}
I_{\tau}^{*}(\alpha, \epsilon, \delta) := \left\{\alpha \in (0,1) \; : \; 
  \left|\alpha-\frac{b}{q^d}\right|<\frac{1}{q^{\tau}}\quad \mbox{ i.o. with  } b\in a_d G_d(q) \mbox{ and }   q^{\epsilon}\le \gcd(b,q) < q^{\epsilon + \delta} \right\}.
\end{equation} 

\begin{nota}
Given $\epsilon\in [0, 1]$ and $\delta >0$, $\mathcal{N}\left(Q, \alpha, \epsilon, \delta \right)$ shall denote the counting function of the set $I_{\tau}^{*}(\alpha, \epsilon, \delta)$, which can be defined more conveniently in this case as follows~:
\begin{align}\label{dederniereminute}
&\mathcal{N}\left( Q, \alpha, \epsilon, \delta \right) := \nonumber \\
&\textrm{Card}\left\{q^d\in \llbracket 1 , Q \rrbracket  \; : \; \left|q^d\alpha- b\right|<q^{d-\tau} \; \; \mbox{i.o. with  } b\in a_d G_d(q) \mbox{ and }   q^{\epsilon}\le \gcd(b,q) < q^{\epsilon + \delta} \right\}.
\end{align}
\end{nota}

With these definitions and this notation at one's disposal, one may now state the following lemma.

\begin{lem}\label{estimfcterreuretvidem}
Assume that $\tau \in (d, d+1)$. Then the set $R_{\tau}^{*}(\alpha, \epsilon, \delta)$ is empty for almost all $\alpha\in [0,1]$ if $\epsilon > d+1-\tau$.

Furthermore, if $0\le \epsilon < \epsilon + \delta < d+1-\tau$, then, for almost all $\alpha\in [0, 1]$, $$Q^{d+1-\tau-\epsilon - \delta - \mu} \; \ll \; \mathcal{N}\left(Q, \alpha, \epsilon, \delta \right) \; \ll \; Q^{d+1-\tau-\epsilon + \nu},$$ where $\mu , \nu>0$ are arbitrarily small. 
\end{lem}

\begin{proof}
To demonstrate the first part of the statement, it suffices to prove that the set $I_{\tau}^{*}(\alpha, \epsilon, \delta)$ is empty in the metric sense as soon as $\epsilon > d+1-\tau$. With this goal in mind, define $$B_P\left(q, \epsilon, \delta \right) := \left\{ b \imod{q} \; : \;   b\in a_d G_d(q) \; \mbox{ and } \; q^{\epsilon}\le \gcd(b,q) < q^{\epsilon + \delta} \right\}$$ and 
\begin{equation}\label{jtauetoileqalphaepsilondelta}
J^*_{\tau}(q, \epsilon, \delta) := \bigcup_{\underset{b\in B_P\left(q, \epsilon, \delta \right)}{0\le b \le q^d-1}}\left(\frac{b}{q^d}-\frac{1}{q^{\tau}} , \frac{b}{q^d}+\frac{1}{q^{\tau}}  \right),
\end{equation} 
in such a way that $\cup_{q\ge N} J^*_{\tau}(q, \epsilon, \delta)$ is a cover of $I_{\tau}^{*}(\alpha, \epsilon, \delta)$ for any $N\ge 1$.

Since
\begin{equation*}
\left| B_P\left(q, \epsilon, \delta \right) \right| = \sum_{\underset{q^{\epsilon} \le a < q^{\epsilon + \delta}}{a\mid q}} \textrm{Card} \left\{ b\imod{q} \; : \; b\in a_d G_d(q) \, \mbox{ and } \, \gcd(b,q) = a \right\},
\end{equation*}
it should be clear that 
\begin{equation*}
\left| B_P\left(q, \epsilon, \delta \right) \right| = \sum_{\underset{q^{\epsilon} \le a < q^{\epsilon + \delta}}{a\mid q}} \textrm{Card} \left\{ b\imod{q} \, : \, b\in G_d(q) \mbox{ and } \gcd(b,q) = a \right\}
\end{equation*} 
if $\gcd(a_d, q)=1$ and that
\begin{equation*}
\left| B_P\left(q, \epsilon, \delta \right) \right| \le \sum_{\underset{q^{\epsilon} \le a < q^{\epsilon + \delta}}{a\mid q}} \textrm{Card} \left\{ b\imod{q} \, : \, b\in G_d(q) \mbox{ and } \gcd(b,q) = a \right\}
\end{equation*} 
if $\gcd(a_d, q)>1$.

Now, if $a$ divides $q$, the ring $a\Z/q\Z$ is isomorphic to $\Z/\tilde{q}\Z$, where $\tilde{q} = q/a$. Therefore, for such an integer $a$, 
\begin{equation*}
\textrm{Card} \left\{ b\imod{q} \, : \, b\in G_d(q) \mbox{ and } \gcd(b,q) = a \right\} = \textrm{Card} \left\{ b\imod{\frac{q}{a}} \, : \, b\in G_d\left(\frac{q}{a}\right) \right\} := r_d\left(\frac{q}{a}\right)
\end{equation*}
from the definition of $r_d(n)$ in subsection~\ref{countt}. Therefore,
\begin{equation}\label{cardbpqepsilondelta}
\left| B_P\left(q, \epsilon, \delta \right) \right| = \sum_{\underset{q^{\epsilon} \le a < q^{\epsilon + \delta}}{a\mid q}} r_d\left(\frac{q}{a}\right) = \sum_{\underset{q^{1-\epsilon -\delta} < l \le q^{1-\epsilon}}{l\mid q}} r_d(l)
\end{equation}
if $\gcd(a_d, q)=1$ and 
\begin{equation*}
\left| B_P\left(q, \epsilon, \delta \right) \right|\le \sum_{\underset{q^{\epsilon} \le a < q^{\epsilon + \delta}}{a\mid q}} r_d\left(\frac{q}{a}\right) = \sum_{\underset{q^{1-\epsilon -\delta} < l \le q^{1-\epsilon}}{l\mid q}} r_d(l)
\end{equation*}
if $\gcd(a_d, q)>1$.

From~(\ref{majorationrdq}) and (\ref{rqtilde}), it is readily checked that 
\begin{align}\label{sumrdlepsilondelta}
\frac{q^{1-\epsilon - \delta}}{(4d)^{\omega\left(q\right)}} \; \le \; \sum_{\underset{q^{1-\epsilon -\delta} < l \le q^{1-\epsilon}}{l\mid q}} r_d(l) \; \le \; 2^{\omega\left( q\right)} \tau(q)^2 q^{1-\epsilon} .
\end{align}
Thus, combining~(\ref{jtauetoileqalphaepsilondelta}), (\ref{cardbpqepsilondelta}) and (\ref{sumrdlepsilondelta}), it follows that, if $\gcd(a_d, q)=1$, 
\begin{equation}\label{encadrementlambdajtauetoileepsilondelta}
\frac{2q^{d-\epsilon-\delta}}{(4d)^{\omega\left(q\right)} q^{\tau}}\; \le \; \lambda\left(J^*_{\tau}(q, \epsilon, \delta )\right) = \frac{2\left|B_P\left(q, \epsilon, \delta \right) \right|q^{d-1}}{q^{\tau}} \; \le \; \frac{2.2^{\omega(q)}\tau(q)^2 q^{d-\epsilon}}{q^{\tau}}\cdotp
\end{equation}
On the one hand, Lemmas~\ref{tau} and \ref{omega} imply that the right-hand side of~(\ref{encadrementlambdajtauetoileepsilondelta}) is the general term of a series which converges whenever $\epsilon > 1+d-\tau$, hence, from the convergent part of the Borel--Cantelli Lemma, $\lambda\left(I^*_{\tau}(q, \epsilon, \delta )\right) =0$ as soon as $\epsilon > 1+d-\tau$.

On the other hand, Lemma~\ref{tau}, Lemma~\ref{omega} and~(\ref{encadrementlambdajtauetoileepsilondelta}) also imply that, for any $\mu, \nu >0$, 
\begin{align*}
Q^{1+d-\tau-\epsilon - \delta - \mu}\; \ll \; \sum_{\underset{\gcd(a_d, q)=1}{1\le q \le Q}} \frac{1}{q^{\tau-d+\epsilon+\delta+\mu}}\; &\ll \; \sum_{1\le q \le Q} \lambda\left(J^*_{\tau}(q, \epsilon, \delta )\right) \\ 
&\ll \; \sum_{1\le q \le Q} \frac{1}{q^{\tau-d+\epsilon - \nu}} \; \ll \; Q^{1+d-\tau - \epsilon + \nu}.
\end{align*}
To conclude the proof, it suffices to notice that, if $\mu$ is chosen so small as $1+d-\tau-\epsilon-\delta-\mu >0$, then, from Corollary~\ref{equivnqalpha}, $$\mathcal{N}\left(Q, \alpha, \epsilon, \delta \right) \sim \sum_{1\le q \le Q} \lambda\left(J^*_{\tau}(q, \epsilon, \delta )\right) \quad \mbox{ as } Q \rightarrow +\infty $$ almost everywhere. 
\end{proof}

\begin{coro}\label{majdimhausrtauetoiledeltaepsilon}
Let $\tau\in (d, d+1)$. Assume that $\epsilon$ and $\delta$ are such that $0\le \epsilon < \epsilon + \delta < 1+d-\tau$. Then, for almost all $\alpha\in [0, 1]$, $$\dim R_{\tau}^{*}(\alpha, \epsilon, \delta) \le \frac{1+d-\tau+\delta}{\tau}\cdotp$$
\end{coro}

\begin{proof}
\sloppy By the definition of the set $R_{\tau}^{*}(\alpha, \epsilon, \delta)$ in~(\ref{rtaualhpaetoileepsilondelta}), its $s$--dimensional Hausdorff measure $\mathcal{H}^s \left(R_{\tau}^{*}(\alpha, \epsilon, \delta) \right)$ satisfies the inequality
\begin{equation}\label{shausdmesurertauetoileaplhaespilondelta}
\mathcal{H}^s \left(R_{\tau}^{*}(\alpha, \epsilon, \delta) \right) \le \sum_{q\ge 1} \sum_{b} \sum_{\underset{b\equiv a_d p^d \imod{q}}{0\le p \le q-1}} \frac{2}{q^{\tau s}},
\end{equation}
where the second sum runs over all the possible integers $b$ such that 
\begin{equation}\label{approxclassiquealphab}
\left|\alpha-\frac{b}{q^d} \right| <\frac{1}{q^{\tau}} \quad \mbox{ and } \quad q^{\epsilon} \le \gcd(b,q) < q^{\epsilon+\delta}.
\end{equation} 
Note that, provided that $q\ge 1$ is large enough and that $\tau>d$, there exists at most one integer $b$ solution to~(\ref{approxclassiquealphab}). So let $\left(q_n\right)_{n\ge 1}$ denote the strictly increasing sequence of denominators $q_n$ such that~(\ref{approxclassiquealphab}) is satisfied for some integer $b_n$. From the definition of this sequence, (\ref{shausdmesurertauetoileaplhaespilondelta}) may be rewritten as
\begin{equation}\label{shausdmesurertauetoileaplhaespilondeltabis}
\mathcal{H}^s \left(R_{\tau}^{*}(\alpha, \epsilon, \delta) \right) \ll \sum_{n\ge 1} \frac{c_n}{q_n^{\tau s}},
\end{equation}
where $c_n := \textrm{Card} \left\{ p \imod{q_n} \; : \; b_n \equiv a_d p^d \imod{q_n} \right\}.$

In order to compute the value of $c_n$, first notice that, from the reasoning developed in Remark~\ref{multirdq}, $c_n$ is multiplicative in $q_n$ (i.e. $c_{nm}=c_nc_m$ whenever $\gcd(q_n, q_m)=1$). Consider now the equation $b_n \equiv a_d p^d \imod{\pi^{\nu_{\pi}(q_n)}},$ where $\pi$ is any prime divisor of $q_n$~:
\begin{itemize}
\item If $b_n \equiv 0 \imod{\pi^{\nu_{\pi}(q_n)}},$ then the equation $a_dp^d \equiv 0 \imod{\pi^{\nu_{\pi}(q_n)}}$ amounts to the following one~: $d\nu_{\pi}(p)+\nu_{\pi}(a_d) \ge \nu_{\pi}(q_n)$. It is readily checked that the number of solutions in $p \imod{q_n}$ to this equation is $$\pi^{\nu_{\pi}(q_n)} - \pi^{k_d(n,\pi)}, \; \textrm{ where }\; k_d(n,\pi):= \left\lceil \frac{\left(\nu_{\pi}(q_n) - \nu_{\pi}(a_d)\right)_+}{d}\right\rceil .$$ 

\item If $b_n \not\equiv 0 \imod{\pi^{\nu_{\pi}(q_n)}},$ then the equation $b_n \equiv a_dp^d \imod{\pi^{\nu_{\pi}(q_n)}}$ amounts to $$p^d \equiv \frac{b_n}{\pi^{\nu_{\pi}(a_d)}} \left(\frac{a_d}{\pi^{\nu_{\pi}(a_d)}}\right)^{-1} \imod{\pi^{\nu_{\pi}(q_n)-\nu_{\pi}(a_d)}},$$ where the division by $\pi^{\nu_{\pi}(a_d)}$ denotes ordinary integer division while multiplicative inversion is performed in $\Z/ \left(\pi^{\nu_{\pi}(q_n)-\nu_{\pi}(a_d)}\right)\Z.$ Using the terminology introduced in Remark~\ref{remsolutionclasse}, the class of any solution $p \imod{\pi^{\nu_{\pi}(q_n)-\nu_{\pi}(a_d)}}$ to this equation has to be $(\nu_{\pi}(b_n)- \nu_{\pi}(a_d))/d$. Therefore, from Remark~\ref{remsolutionclasse}, the number of solutions in $p \imod{\pi^{\nu_{\pi}(q_n)-\nu_{\pi}(a_d)}}$ to this equation is $$u_d\left(\pi^{\nu_{\pi}(q_n)-\nu_{\pi}(a_d) - d\frac{\nu_{\pi}(b_n)-\nu_{\pi}(a_d)}{d}} \right) = u_d\left(\pi^{\nu_{\pi}(q_n) - \nu_{\pi}(b_n)} \right) \le u_d\left(\pi^{\nu_{\pi}(q_n)}\right)$$ (see Proposition~\ref{decompte} for this last inequality).
\end{itemize}
All things considered, 
\begin{equation*}
c_n \; \le \; \prod_{\underset{b_n\equiv 0 \imod{\pi^{\nu_{\pi}\left(q_n\right)}}}{\pi\mid q_n}} \pi^{\nu_{\pi}\left(q_n\right)} \prod_{\underset{b_n \not\equiv 0 \imod{\pi^{\nu_{\pi}\left(q_n\right)}}}{\pi\mid q_n}} u_d\left(\pi^{\nu_{\pi}\left(q_n\right)}\right) \; \le \; \gcd(b_n, q_n)\, u_d\left(q_n\right).
\end{equation*}
Now, from the definition of the set $R_{\tau}^{*}(\alpha, \epsilon, \delta)$, it may be assumed that $\gcd(b_n, q_n) \le q_n^{\epsilon + \delta}$. Therefore, using~(\ref{majudq}) and Lemma~\ref{omega}, it is readily seen that~(\ref{shausdmesurertauetoileaplhaespilondeltabis}) implies that 
\begin{equation}\label{smesrtauetoilalphaepsdetl}
\mathcal{H}^s \left(R_{\tau}^{*}(\alpha, \epsilon, \delta) \right) \ll \sum_{n\ge 1} \frac{1}{q_n^{\tau s -\epsilon - \delta - \gamma}}
\end{equation} 
for arbitrarily small $\gamma >0$. Since $\mathcal{N}\left(q_n, \alpha, \epsilon, \delta \right) = n$ by the definition of the sequence $\left(q_n\right)_{n\ge 1}$, Lemma~\ref{estimfcterreuretvidem} leads to the estimate $$n^{1/(d+1-\tau -\epsilon + \gamma)} \; \ll \; q_n$$ valid for almost all $\alpha\in [0,1]$ and for arbitrarily small $\gamma >0$. Thus, $$\mathcal{H}^s \left(R_{\tau}^{*}(\alpha, \epsilon, \delta) \right) \; \ll \; \sum_{n\ge 1}\frac{1}{n^{(\tau s - \epsilon - \delta -\gamma)/(d+1-\tau -\epsilon +\gamma)}},$$ which is a convergent series for $s\ge (d+1-\tau +\delta +2\gamma)/\tau$, that is, $$\dim R_{\tau}^{*}(\alpha, \epsilon, \delta) \le \frac{d+1-\tau +\delta +2\gamma}{\tau}\cdotp$$ The result follows on letting $\gamma$ tend to zero.
\end{proof}

The proofs of Lemma~\ref{estimfcterreuretvidem} and Corollary~\ref{majdimhausrtauetoiledeltaepsilon} rely strongly on the fact that, when $\epsilon < 1+d-\tau$, it is always possible to choose $\delta>0$ and $\mu>0$ so small as $1+d-\tau-\epsilon - \delta - \mu >0$. While Lemma~\ref{estimfcterreuretvidem} also implies that the set $R_{\tau}^{*}(\alpha, \epsilon, \delta)$ is empty in the metric sense whenever $\epsilon > 1+d-\tau$, this leaves a gap corresponding to the case where $\epsilon = 1+d-\tau$. This limit case is now studied.

Since $R_{\tau}^{*}(\alpha, \epsilon, \delta) = \emptyset$ for almost all $\alpha\in [0,1]$ when $\epsilon > 1+d-\tau$, it should be clear that $R_{\tau}^{*}(\alpha, 1+d-\tau, \delta) = R_{\tau}^{*}(\alpha, 1+d-\tau, \mu)$ for any $\delta, \mu>0$, the equality holding true in the metric sense. Denote by $R_{\tau}^{*}(\alpha, 1+d-\tau)$ the common set determined by these different values of $\delta>0$ and $\mu >0$, i.e. $$R_{\tau}^{*}(\alpha, 1+d-\tau) := \bigcap_{\delta >0} R_{\tau}^{*}(\alpha, 1+d-\tau, \delta).$$ In other words, $R_{\tau}^{*}(\alpha, 1+d-\tau) = R_{\tau}^{*}(\alpha, 1+d-\tau, \delta)$ for any $\delta>0$ and for almost all $\alpha\in [0,1]$. In a similar way, let $$I_{\tau}^{*}(\alpha, 1+d-\tau) := \bigcap_{\delta >0} I_{\tau}^{*}(\alpha, 1+d-\tau, \delta).$$ Thus, $I_{\tau}^{*}(\alpha, 1+d-\tau)$ is to $R_{\tau}^{*}(\alpha, 1+d-\tau)$ as $I_{\tau}^{*}(\alpha, \epsilon, \delta)$ is to $R_{\tau}^{*}(\alpha, \epsilon , \delta)$ when $0\le \epsilon < \epsilon +\delta < 1+d-\tau$, these last two sets having been defined by~(\ref{rtaualhpaetoileepsilondelta}) and~(\ref{Ietoiletauepsilondelta}).

\begin{nota}
\sloppy The quantity $\mathcal{N}\left(Q, \alpha, 1+d-\tau \right)$ shall denote the counting function of the set \mbox{$I_{\tau}^{*}(\alpha,1+d-\tau)$} defined in a similar way as in~(\ref{dederniereminute}).
\end{nota}

As might be expected, the asymptotic behavior of the function $\mathcal{N}\left(Q, \alpha, 1+d-\tau \right)$ is different from that of $\mathcal{N}\left(Q, \alpha, \epsilon, \delta\right)$ when $0\le \epsilon <\epsilon + \delta < 1+d-\tau$~:

\begin{lem}\label{fctcomptagecritique}
Assume that $\tau\in (d, d+1)$. Then for almost all $\alpha\in\R$, $$\mathcal{N}\left(Q, \alpha, 1+d-\tau \right) \; \ll \; Q^{\mu},$$ where $\mu >0$ is arbitrarily small.
\end{lem}

\begin{proof}
\sloppy Let $\delta>0$. Define $J^*_{\tau}(q, 1+d-\tau, \delta)$ as in~(\ref{jtauetoileqalphaepsilondelta}). Then the upper bound for \mbox{$\lambda\left(J^*_{\tau}(q, 1+d-\tau, \delta) \right)$} provided by~(\ref{encadrementlambdajtauetoileepsilondelta}) still holds true, namely $$\lambda\left(J^*_{\tau}(q, 1+d-\tau, \delta) \right) \; \le \; \frac{2.2^{\omega(q)}\tau(q)^2}{q}\cdotp$$ Therefore, since $2^{\omega(q)} = o\left(q^{\mu}\right)$ and $\tau(q) = o\left(q^{\mu}\right)$ for any $\mu>0$ from Lemmas~\ref{tau} and \ref{omega}, for all $Q\ge 1$,
\begin{equation*}
\Phi(Q) := \sum_{q=1}^{Q} \lambda\left(J^*_{\tau}(q, 1+d-\tau, \delta) \right) \; \ll \; \sum_{q=1}^{Q} \frac{1}{q^{1-3\mu}} \; \ll \; Q^{3\mu}
\end{equation*}
and, in a similar way,
\begin{equation*}
\Psi(Q) := \sum_{q=1}^{Q} \lambda\left(J^*_{\tau}(q, 1+d-\tau, \delta) \right) \tau(q) \; \ll \; Q^{4\mu}.
\end{equation*}
Now, $\cup_{q\ge N}J^*_{\tau}(q, 1+d-\tau, \delta)$ is a cover of $I_{\tau}^{*}(\alpha, 1+d-\tau, \delta)$ for any $N\ge 1$. Since the latter set is equal to $I_{\tau}^{*}(\alpha, 1+d-\tau)$  for almost all $\alpha\in [0,1]$, it follows from Theorem~\ref{estimnaalphaerreur} that, for almost all $\alpha\in [0,1]$, $$\mathcal{N}\left(Q, \alpha, 1+d-\tau \right) \; = \; \Phi (Q) + O\left(\sqrt{\Psi(Q)} \left(\log \Psi (Q) \right)^{3/2+\kappa} \right) \; \ll \; Q^{3\mu},$$ where $\kappa>0$ has been chosen arbitrarily.
\end{proof}

\begin{coro}\label{majdimhausdrtauetoilealpheundtau}
Assume that $\tau\in (d, d+1)$. Then, for almost all $\alpha\in [0,1]$, $$\dim R_{\tau}^{*}(\alpha, 1+d-\tau) \le \frac{1+d-\tau}{\tau}\cdotp$$
\end{coro}

\begin{proof}
Let $\delta >0$. Denote by $\left(q_n\right)_{n\ge 1}$ the strictly increasing sequence of denominators $q_n$ such that~(\ref{approxclassiquealphab}) with $\epsilon = 1+d-\tau$ is satisfied for some integer $b_n$. Then inequality~(\ref{smesrtauetoilalphaepsdetl}) still holds true for the set $R_{\tau}^{*}(\alpha, 1+d-\tau, \delta)$, namely 
\begin{equation*}
\mathcal{H}^s \left(R_{\tau}^{*}(\alpha, 1+d-\tau, \delta) \right) \ll \sum_{n\ge 1} \frac{1}{q_n^{\tau s -1-d+\tau - \delta - \gamma}}
\end{equation*} 
for arbitrarily small $\gamma >0$.

Since $I_{\tau}^{*}(\alpha, 1+d-\tau, \delta) = I_{\tau}^{*}(\alpha, 1+d-\tau)$ for almost all $\alpha\in [0,1]$, the counting functions of these two sets have the same asymptotic behaviour, hence, from Lemma~\ref{fctcomptagecritique}, $$n\;= \; \mathcal{N}\left(q_n, \alpha, 1+d-\tau \right) \; \ll \; q_n^{\mu}$$ almost everywhere, with $\mu>0$ arbitrary. Therefore, for almost all $\alpha\in [0,1]$, $$\mathcal{H}^s \left(R_{\tau}^{*}(\alpha, 1+d-\tau) \right) \; \ll \; \sum_{n\ge 1}\frac{1}{n^{(\tau s - 1 - d + \tau - \delta -\gamma)/\mu}},$$ which is a convergent series for $s\ge (1+d-\tau +\delta +\gamma +\mu)/\tau$, that is, $$\dim R_{\tau}^{*}(\alpha, 1+d-\tau) \le \frac{d+1-\tau +\delta +\gamma + \mu}{\tau}\cdotp$$ The result follows on letting $\gamma$, $\delta$ and $\mu$ tend to zero.
\end{proof}

\begin{proof}[Completion of the proof of Theorem~\ref{sec}.]
In order to prove that $\dim R_{\tau}^{*}(\alpha) \le (d+1-\tau)/\tau$ for almost all $\alpha\in [0,1]$ when $\tau\in (d, d+1)$, recall first that, from Lemma~\ref{estimfcterreuretvidem}, the equality
\begin{equation*}
R_{\tau}^{*}(\alpha) \; = \; \left(\bigcup_{0\le \epsilon < \epsilon +  \delta < 1+d-\tau} R_{\tau}^{*}(\alpha, \epsilon, \delta) \right) \; \cup \; R_{\tau}^{*}(\alpha, 1+d-\tau) 
\end{equation*} 
holds true almost everywhere. Corollary~\ref{majdimhausdrtauetoilealpheundtau} also implies that it suffices to prove that $$\dim \left(\bigcup_{0\le \epsilon < \epsilon +  \delta < 1+d-\tau} R_{\tau}^{*}(\alpha, \epsilon, \delta) \right) \le \frac{1+d-\tau}{\tau}$$ for almost all $\alpha\in [0,1]$ when $\tau$ lies in the interval $(d, d+1)$.

To this end, consider a strictly increasing sequence $\left(\beta_p\right)_{p\ge 0}$ of real numbers from the interval $(0, 1+d-\tau)$ tending to $1+d-\tau$ as $p$ tends to infinity. It should then be obvious that 
\begin{equation*}
\bigcup_{0\le \epsilon < \epsilon +  \delta < 1+d-\tau} R_{\tau}^{*}(\alpha, \epsilon, \delta) = \bigcup_{p\ge 0} R_{\tau}^{*} (\alpha)[p], \quad \textrm{ where } \quad R_{\tau}^{*}(\alpha)[p]:= \bigcup_{0\le \epsilon < \epsilon +  \delta \le \beta_p} R_{\tau}^{*}\left(\alpha, \epsilon, \delta \right).
\end{equation*}
Given $p\in\N$, let $\left(\epsilon_p(k)\right)_{0\le k \le n}$ be the finite sequence subdividing the interval $\left[0 , \beta_p\right]$ into $n\ge 1$ intervals of equal length $\delta_p(n)$ and satisfying $$\epsilon_p(0)=0 \quad \textrm{ and } \quad \epsilon_p(n) = \beta_p = \epsilon_p(n-1)+\delta_p(n).$$
Then $$R_{\tau}^{*}(\alpha)[p] = \bigcup_{k=0}^{n-1} R_{\tau}^{*}\left(\alpha, \epsilon_p(k), \delta_p(n) \right)$$ for any regular subdivision $\left(\epsilon_p(k)\right)_{0\le k \le n}$ of $[0, \beta_p]$ into $n\ge 1$ intervals. Thus, from Corollary~\ref{majdimhausrtauetoiledeltaepsilon}, $$\dim R_{\tau}^{*}(\alpha)[p] = \underset{0\le k \le n-1}{\sup} \dim R_{\tau}^{*}\left(\alpha, \epsilon_p(k), \delta_p(n) \right) \le \frac{d+1-\tau+\delta_p(n)}{\tau},$$ which holds true for almost all $\alpha\in [0,1]$ and for any $n\ge 1$. On letting $n$ tend to infinity, $\delta_p(n)$ tends to zero and it follows that, outside a zero Lebesgue measure set, $$\dim R_{\tau}^{*}(\alpha)[p] \le \frac{1+d-\tau}{\tau}$$ when $\tau\in (d, d+1)$. Since the set $R_{\tau}^{*}(\alpha)$ is the countable union of $R_{\tau}^{*}(\alpha, 1+d-\tau)$ and of $R_{\tau}^{*}(\alpha)[p]$ ($p\in\N$) for almost all $\alpha\in [0,1]$, this completes the proof.
\end{proof}

\section{Concluding Remarks} \label{sec6}

Some remarks on the method developed in this paper and the relevance of the results obtained are stated to conclude. 

\begin{itemize}
\item The upper bound for the Hausdorff dimension of the set $W_{\tau}\left(P_{\alpha}\right)$ stated in Theorem~\ref{sec} is easily seen to be non--optimal as soon as $\tau < d-1$ as it is superseded by the Hausdorff dimension of the set of $\tau$--well approximable numbers given by the Theorem of Jarn\'ik  and Besicovitch~: if $W_{\tau}$ denotes the latter set, then $W_{\tau}\left(P_{\alpha}\right)\subset W_{\tau}$ for any $\tau>0$ and $\dim W_{\tau} = 2/\tau$ whenever $\tau > 2$. 

Now if $\tau\in [d-1, d]$, the study of the case $d=3$ also tends to provide evidence that the upper bound $(1+d-\tau)/\tau$ is still not relevant in the general case. Indeed, when $d=3$, on letting $\tau$ tend to $2$ from above (resp. from below) in Theorem~\ref{sec} (resp. in Theorem~\ref{bdv}), the upper bound thus found for $\underset{\tau\rightarrow 2^+}{\lim} \dim W_{\tau}\left(P_{\alpha}\right)$ is clearly seen to be non--optimal.

More generally, the actual Hausdorff dimension of the set of $\tau$--well approximable points lying on a polynomial curve when $\tau $ is larger than 2 and less than the degree of the polynomial remains an open problem for which nothing is known (see also \cite{base} for another mention of this problem).

\item As mentioned in the introduction of this paper, the upper bound for $\dim W_{\tau}\left(P_{\alpha}\right)$ given by Theorem~\ref{sec} is more than likely the actual value for the Hausdorff dimension of $W_{\tau}\left(P_{\alpha}\right)$ for almost all $\alpha\in [0,1]$ when $\tau$ lies in the interval $(d, d+1]$. To also obtain $(d+1-\tau)/\tau$  as a lower bound for $\dim W_{\tau}\left(P_{\alpha}\right)$, it would be sufficient to prove such a result for the set $\tilde{R}_{\tau}\left( \alpha \right)$ as defined in~(\ref{ensRdetravail}). However, this would imply the study of the distribution of solutions to congruence equations and a quantitative result on the uniformity of such a distribution. For arbitrary polynomials, this appears to be out of reach at the moment.

\item The set of exceptions (with respect to $\alpha$) left by Theorem~\ref{princi} actually contains uncountably many points. Indeed, let $\tau>d+1$ be given and let $(x,y)$ be a pair of real numbers simultaneously $\tau$--well approximable --- this set is uncountable as its Hausdorff dimension is $3/\tau$ from the multidimensional generalization of the Theorem of Jarn\'ik  and Besicovitch. Then, setting $\alpha = y-P(x)$, it is readily seen that $x$ lies in $W_{\tau}\left(P_{\alpha}\right)$ since $x$ and $P(x)+\alpha$ are simultaneously $\tau$--well approximable.
\end{itemize}

\renewcommand{\abstractname}{Acknowledgements}
\begin{abstract}
The author would like to thank his PhD supervisor Detta Dickinson for suggesting the problem and for discussions which helped to develop ideas put forward. The last concluding remark is entirely due to her. He is supported by the Science Foundation Ireland grant RFP11/MTH3084.
\end{abstract}

\maketitle\bibliographystyle{plain}
\bibliography{reference}
\end{document}